\documentclass{article}
\usepackage{amsfonts, mathrsfs, bm, amsmath, amssymb, subfigure, slashbox}
\usepackage{graphicx}

\newtheorem{alg}{Algorithm}[section]
\newtheorem{definition}{Definition}[section]
\newtheorem{theorem}{Theorem}[section]
\newtheorem{corollary}{Corollary}[section]
\newtheorem{lemma}{Lemma}[section]

\newenvironment{proof}{\paragraph{Proof:}}{\hfill$\square$}

\begin{document}


\title{A Two-Dimensional Improvement for Farr-Gao Algorithm\footnote{Supported by National Natural Science Foundation of China (No. 11101185).}}

\author{Tian Dong\footnote{email: dongtian@jlu.edu.cn}\\ \small School of Mathematics, Key Lab. of Symbolic Computation\\ \small and Knowledge Engineering \textup{(}Ministry of Education\textup{)}, \\ \small Jilin University, Changchun 130012, China}

\date{}

%

\maketitle


\begin{abstract}
Farr-Gao algorithm is a state-of-the-art algorithm for reduced Gr\"{o}bner bases of vanishing ideals of finite points, which has been implemented in \verb"Maple"$^\circledR$ as a build-in command. In this paper, we present a two-dimensional improvement for it that employs a preprocessing strategy for computing reduced Gr\"{o}bner bases associated with tower subsets of given point sets. Experimental results show that the preprocessed Farr-Gao algorithm is more efficient than the classical one.

\vskip 6pt

\noindent \textbf{Keywords}: Gr\"{o}bner basis; Vanishing ideal; Tower set; Gr\"{o}bner \'{e}scalier; Newton interpolation basis

\vskip 6pt

\noindent \textbf{2000 Mathematics Subject Classification}: 13P10
\end{abstract}

\section{Introduction}
\label{s:int}

Let $\mathbb{F}$ be a field, and let $\Pi^d:=\mathbb{F}[x_1, x_2, \ldots, x_d]$ denote the $d$-variate polynomial ring over $\mathbb{F}$. It is well known that the set of polynomials in $\Pi^d$ that vanish at a finite nonempty set $\Xi\subset \mathbb{F}^d$ forms an ideal in $\Pi^d$ which is called the \emph{vanishing ideal} of $\Xi$, denoted by $\mathcal{I}(\Xi)$. In view of the applications of vanishing ideals in the fields of mathematics and other sciences in recent years\cite{GS1999, LS2004}, there has been increasing interest in their
Gr\"{o}bner bases\cite{CLO2007} and Gr\"{o}bner \'{e}scaliers (aka standard set, standard monomials etc.)\cite{Mor2009}.

The most significant milestone of computing vanishing ideals
is the Buchberger -M\"{o}ller algorithm \cite{MB1982} that yields, for fixed $\Xi$ and monomial order $\prec$ on $\Pi^d$, the
reduced Gr\"{o}bner basis $G$ and the Gr\"{o}bner \'{e}scalier $M$ of $\mathcal{I}(\Xi)$ w.r.t. $\prec$. It also produces a Newton interpolation basis $N$ for the $\mathbb{F}$-linear space spanned by $M$. One decade later, MMM algorithm in \cite{MMM1993} extended Buchberger-M\"{o}ller algorithm to solve general zero-dimensional ideals. And then, \cite{ABKR2000}
introduced a modular version of Buchberger-M\"{o}ller algorithm with lower complexity in $\mathbb{Q}^d$. All these three algorithms apply Gauss elimination on generalized Vandermonde matrices regardless of the order of the points in $\Xi$.

As is well known, $G$ depends not only on $\#\Xi$ (as in univariate cases) but significantly on the geometry (distribution) of $\Xi$ that is complex in $\mathbb{F}^d$ (see \cite{Sau2006}). However, the algorithms mentioned above do not take this into account. In 2006, Farr and Gao \cite{FG2006} presented a more effcient algorithm for $G$ (called Farr-Gao algorithm hereafter) that has been implemented in \verb"Maple"$^\circledR$ as build-in command \verb"VanishingIdeal". Arguably, the most distinguishing feature of Farr-Gao algorithm is its point-sorting strategy that provides the possibility to borrow the idea of univariate Newton interpolation. Once the points are sorted, the computation will be performed along parallel lines, one after another. On each line, we are essentially solving univariate Newton interpolation and hence the amount of reduction is decreased. The process of Farr-Gao algorithm implies that multivariate Newton interpolation would be helpful for the computation of vanishing ideals. Concretely, if we could theoretically obtain a Newton interpolation basis of some subset of $\Xi$, then the amount of reduction of Farr-Gao algorithm would be decreased further.

Let $\Xi$ be a Cartesian set in $\mathbb{F}^d$. \cite{Sau2004} gave the unique Gr\"{o}bner \'{e}scalier $M$ of $\mathcal{I}(\Xi)$ in theory which implies that $\mathcal{I}(\Xi)$ has a unique reduced Gr\"{o}bner bases  w.r.t. any monomial order (see \cite{LZD2012}). Moreover, we can also construct Newton interpolation bases for $\Xi$ theoretically. Based on this, \cite{WZD2010} proposed a bivariate preprocessing paradigm for Buchberger-M\"{o}ller algorithm that inputs the monomial (Gr\"{o}bner \'{e}scalier) and Newton interpolation basis for a maximal Cartesian subset of $\Xi$ into Buchberger-M\"{o}ller algorithm as initial values. However, since the distribution of a Cartesian set is fairly restricted, in many cases the maximal Cartesian subsets are not large enough and therefore the improvement is minor.

In the following, we first introduce a new type of finite nonempty sets, tower sets, in $\mathbb{F}^2$ that have ``freer" distributions than bivariate Cartesian sets whose formal definition is provided in Section \ref{s:tower} where we also establish a new criterion for bivariate Cartesian sets for the purpose of investigating the relation between tower sets and Cartesian sets. Next, in Section \ref{s:result}, we theoretically offer the Gr\"{o}bner \'{e}scaliers of vanishing ideals of tower sets w.r.t. commonly used monomial orders as well as the Newton interpolation bases spanned by them. And, finally, these results lead to our main algorithm and the timings of some experiments are given.

\section{Bivariate tower sets}\label{s:tower}

Let $\mathbb{N}_0$ stand for the monoid of nonnegative integers. A \emph{polynomial}
$f\in \Pi^2=\mathbb{F}[x, y]$ is of the form
$$
f=\sum_{(i,j)\in \mathbb{N}_0^2} c_{ij} x^iy^j, \hspace{0.8cm}\#\{(i,j)\in
\mathbb{N}_0^2: 0\neq c_{ij}\in \mathbb{F} \} < \infty,
$$
where \emph{monomial} $x^iy^j$ is a product for vector $(i, j)$. The set of all monomials in $\Pi^2$ is denoted by $\mathbb{T}^2$.

Fix a monomial order $\prec$ on $\Pi^2$ that could be lexicographical
order $\prec_\text{lex}$ (\verb"plex(x, y)" in \verb"Maple"$^\circledR$), inverse lexicographical order
$\prec_\text{inlex}$ (\verb"plex(y, x)" in \verb"Maple"$^\circledR$), graded lexicographical order
$\prec_\text{grlex}$ (\verb"grlex(x, y)" in \verb"Maple"$^\circledR$), or graded reverse lexicographical order
$\prec_\text{grevlex}$ (\verb"tdeg(x, y)" in \verb"Maple"$^\circledR$) etc, cf. \cite{CLO2007}. For all nonzero $f
\in \Pi^2$, we let $\mathrm{LT}_\prec(f)$ signify the \emph{leading term}, $\mathrm{LM}_\prec(f)$ the \emph{leading monomial}, and $\mathrm{LC}_\prec(f)$
the \emph{leading coefficient} of $f$. Furthermore, for a nonempty
set $F \subset \Pi^2$, set
$$
    \mathrm{LM}_\prec(F) := \{ \mathrm{LM}_\prec(f) : f \in F\}.
$$
Let $G$ be the reduced Gr\"{o}bner basis for a zero-dimensional ideal $\mathcal{I}\subset \Pi^2$ w.r.t. $\prec$. According to \cite{Mor2009}, the monomial set
\begin{equation}\label{e:nf}
  \mathcal{N}_\prec(\mathcal{I}):=\{t\in \mathbb{T}^2:
\mathrm{LM}_\prec(g)\nmid t, \forall g\in G\}
\end{equation}
forms the \emph{Gr\"{o}bner \'{e}scalier} of $\mathcal{I}$ w.r.t. $\prec$, and its
\emph{corner}
\begin{equation}\label{e:corner}
        \mathcal{C}[\mathcal{N}_\prec(\mathcal{I})]:=\left\{t\in \mathbb{T}^2: x|t\Rightarrow
t/x\in \mathcal{N}_\prec(\mathcal{I}), y|t\Rightarrow
t/y\in \mathcal{N}_\prec(\mathcal{I})\right\}\backslash \mathcal{N}_\prec(\mathcal{I})
\end{equation}
is equal to $\mathrm{LM}_\prec(G)$.

Let $\mathcal{A}$ be a finite nonempty set in $\mathbb{N}_0^2$. It is called \emph{lower} if for any $(i, j) \in \mathcal{A}$ we always have
\begin{equation}\label{e:lower}
\{(i', j')\in \mathbb{N}_0^2: 0\leq i' \leq i, 0\leq j' \leq j\} \subseteq \mathcal{A}.
\end{equation}
Set
\begin{equation}\label{e:bb}
    \mathrm{b}_x(\mathcal{A}):=\max_{(i, 0)\in \mathcal{A}} i,\qquad \mathrm{b}_y(\mathcal{A}):=\max_{(0, j)\in \mathcal{A}} j.
\end{equation}
Subfigure (b) of Fig. \ref{f:1} illustrates a lower set with $(\mathrm{b}_x, \mathrm{b}_y)=(4, 7)$ labeled. Obviously, $\mathrm{b}_x(\mathcal{A})$ and $\mathrm{b}_y(\mathcal{A})$ alone are not enough to determine $\mathcal{A}$. Hence, we introduce sequences $m_0, m_1, \ldots, m_{\mathrm{b}_y(\mathcal{A})}$ and $n_0, n_1, \ldots, n_{\mathrm{b}_x(\mathcal{A})}$ that can uniquely determine $\mathcal{A}$ respectively, where
\begin{align*}
    m_j=&\max_{(i,j)\in \mathcal{A}}i,\quad 0\leq j\leq \mathrm{b}_y(\mathcal{A}),\\
    n_i=&\max_{(i,j)\in \mathcal{A}}j,\quad 0\leq i\leq \mathrm{b}_x(\mathcal{A}).
\end{align*}
It is easy to see that \eqref{e:bb} implies $m_0=\mathrm{b}_x(\mathcal{A})$ and $n_0=\mathrm{b}_y(\mathcal{A})$. Thus, it makes sense to write
\begin{equation}\label{e:lowerpara}
  \mathcal{A}=\mathrm{L}_x(m_0,m_1,\ldots,m_{\mathrm{b}_y(\mathcal{A})})=\mathrm{L}_y(n_0, n_1, \ldots, n_{\mathrm{b}_x(\mathcal{A})}).
\end{equation}
A simple observation shows that the lower set in Subfigure (b) of Fig. \ref{f:1} is
$$
\mathrm{L}_x(4, 3, 3, 2, 1, 0, 0, 0)=\mathrm{L}_y(7, 4, 3, 2, 0).
$$
Moreover, from \eqref{e:lower} we can deduce that both $m_0, \ldots, m_{\mathrm{b}_y(\mathcal{A})}$ and $n_0, \ldots, n_{\mathrm{b}_x(\mathcal{A})}$ are monotonically decreasing sequences. Furthermore, if they are strictly monotonically decreasing, then we say that $\mathcal{A}$ is $x$-\emph{strict} (resp. $y$-\emph{strict}) lower.

As index sets, the lower sets in $\mathbb{N}_0^2$ are used to label Cartesian sets in $\mathbb{F}^2$ as follows.

\begin{definition}\textup{\cite{LZD2012}}\label{d:lowerset}
A finite nonempty set
$\Xi\subset \mathbb{F}^2$ of distinct points is \emph{Cartesian} if and only if there exists
a lower set $\mathcal{A}\subset \mathbb{N}_0^2$ and two injective functions $\mathrm{x, y}:\mathbb{N}_0\rightarrow \mathbb{F}$ such that $\Xi$ can be written as
\begin{equation}\label{e:lowerset}
\Xi=\left\{\left(\mathrm{x}(i), \mathrm{y}(j)\right)\in \mathbb{F}^2: (i, j)\in \mathcal {A}\right\}.
\end{equation}
$\Xi$ is also called $\mathcal{A}$-\emph{Cartesian}.
\end{definition}

Subfigure (a) of Fig. \ref{f:1} illustrates a Cartesian set that is labelled by the lower set mentioned above.

Given a finite nonempty set $\Xi\subset \mathbb{F}^2$ of distinct points. Set
$$
\pi_x(\Xi):=\{\pi_x(\xi): \xi\in \Xi\}\quad\mbox{ and }\quad \pi_y(\Xi):=\{\pi_y(\xi): \xi\in \Xi\}
$$
as the first and the second projection maps on $\mathbb{F}^2$ respectively, namely
$$
    \pi_x: \mathbb{F}^2 \rightarrow \mathbb{F}: (x, y)\mapsto x\quad\mbox{ and }\quad
    \pi_y: \mathbb{F}^2 \rightarrow \mathbb{F}: (x, y)\mapsto y.
$$

Recall the point-sorting strategy of Farr-Gao algorithm, $\Xi$ can be decomposed vertically and horizontally as
\begin{equation}\label{e:decom}
  \begin{aligned}
    \Xi=&\bigcup_{\bar{x}\in\pi_x(\Xi)} \Xi\cap \{x=\bar{x}\}=:\bigcup_{i=0}^{\#\pi_x(\Xi)-1} \Xi_i^y\\
       =&\bigcup_{\bar{y}\in\pi_y(\Xi)} \Xi\cap \{y=\bar{y}\}=:\bigcup_{j=0}^{\#\pi_y(\Xi)-1} \Xi_j^x,
\end{aligned}
\end{equation}
where $\#\Xi_0^y\geq\cdots\geq \#\Xi_{\#\pi_x(\Xi)-1}^y$ and $\#\Xi_0^x\geq\cdots\geq \#\Xi_{\#\pi_y(\Xi)-1}^x$. Subfigure (a) of Fig.\ref{f:1} displays the decompositions of a Cartesian set.

In \cite{Cra2004}, two particular lower sets in $\mathbb{N}_0^2$
\begin{equation}\label{e:sxsy}
  \begin{aligned}
  S_x(\Xi):=&\mathrm{L}_x(\#\Xi_0^x-1, \ldots, \#\Xi_{\#\pi_y(\Xi)-1}^x-1),\\
  S_y(\Xi):=&\mathrm{L}_y(\#\Xi_0^y-1, \ldots, \#\Xi_{\#\pi_x(\Xi)-1}^y-1)
\end{aligned}
\end{equation}
are constructed from $\Xi$ (see (b) of Fig.\ref{f:1} for example), which reflect the distribution of $\Xi$ in certain sense, and the following criterion for Cartesian sets in $\mathbb{F}^2$ is offered as well.

\begin{figure}[!htbp]
\centering
\subfigure[$\Xi=\bigcup_{i=0}^4 \Xi_i^y=\bigcup_{j=0}^7 \Xi_j^x$]{\includegraphics[width=6cm,height=6cm]{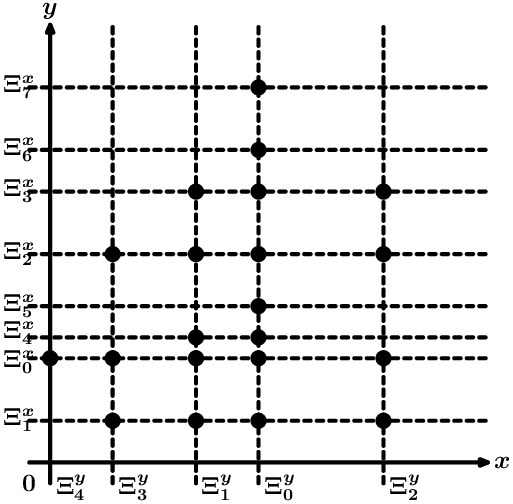}}
\subfigure[$S_y(\Xi)$]{\includegraphics[width=6cm,height=6cm]{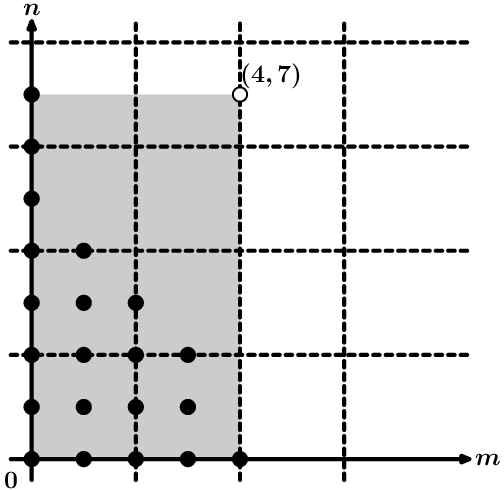}}

\caption{Cartesian set $\Xi$, its decompositions, and $S_y(\Xi)$}\label{f:1}  
\end{figure}

\begin{theorem}\textup{\cite{Cra2004}}\label{t:lowercar}
A finite nonempty set $\Xi\subset \mathbb{F}^2$ is Cartesian if and only if
$S_x(\Xi)=S_y(\Xi)$.
\end{theorem}

Theorem \ref{t:lowercar}, Definition \ref{d:lowerset}, and \eqref{e:sxsy} yield the following corollary immediately.

\begin{corollary}\label{c:asxsy}
If a finite nonempty set $\Xi\subset \mathbb{F}^2$ is $\mathcal{A}$-Cartesian, then $\mathcal{A}=S_x(\Xi)=S_y(\Xi)$. Consequently, \eqref{e:lowerset} can be rewritten as
\begin{align}
    \Xi=&\left\{\left(\mathrm{x}(i), \mathrm{y}(j)\right)\in \mathbb{F}^2: (i, j)\in S_x(\Xi)\right\}\label{e:newlowersetx}\\
       =&\left\{\left(\mathrm{x}(i), \mathrm{y}(j)\right)\in \mathbb{F}^2: (i, j)\in S_y(\Xi)\right\}\label{e:newlowersety}.
\end{align}

\end{corollary}

Unfortunately, it is difficult to extend Theorem \ref{t:lowercar} to three and higher dimensions. Therefore, we give the following criterion that extends to higher dimensions naturally.

\begin{theorem}\label{t:lower}
A finite nonempty set $\Xi\subset \mathbb{F}^2$ with decompositions \eqref{e:decom} is
Cartesian if and only if
\begin{equation}\label{e:pix}
    \pi_x(\Xi_0^x) \supseteq \pi_x(\Xi_1^x) \supseteq \cdots \supseteq \pi_x(\Xi_{\#\pi_y(\Xi)-1}^x)
\end{equation}
or
\begin{equation}\label{e:piy}
  \pi_y(\Xi_0^y) \supseteq \pi_y(\Xi_1^y) \supseteq \cdots \supseteq \pi_y(\Xi_{\#\pi_x(\Xi)-1}^y).
\end{equation}
\end{theorem}

\begin{proof}
  Assume that $\Xi$ is $\mathcal{A}$-Cartesian satisfying \eqref{e:lowerset} where $\mathcal{A}$ can be represented as \eqref{e:lowerpara} that together with \eqref{e:lowerset} implies $\mathrm{b}_y(\mathcal{A})=\#\pi_y(\Xi)-1$ and
$$
   \pi_x(\Xi_j^x)=\{\mathrm{x}(0), \mathrm{x}(1), \ldots, \mathrm{x}({m_j})\},\quad j=0, \ldots, b_y(\mathcal{A}).
$$
Since $m_0\geq m_1\geq \cdots\geq m_{b_y(\mathcal{A})}\geq 0$, \eqref{e:pix} follows. \eqref{e:piy} can be proved in like manner.

Conversely, we suppose that \eqref{e:pix} holds and that lower sets $S_x(\Xi)$ and $S_y(\Xi)$ have the expression \eqref{e:sxsy}. Then there exists a unique sequence $(n_0, n_1, \ldots,n_{\#\Xi_0^x-1})\in \mathbb{N}_0^{\#\Xi_0^x}$ such that
\begin{equation}\label{e:sxly}
    S_x(\Xi)=\mathrm{L}_y(n_0, n_1, \ldots,n_{\#\Xi_0^x-1})
\end{equation}
where $n_0=\#\pi_y(\Xi)-1$.

Next, we shall verify that $\#\pi_x(\Xi)=\#\Xi_0^x$, namely we can cover $\Xi$ by exactly
$\#\Xi_0^x$ vertical lines. Prove this by contradiction. It is evident from $\#\pi_x(\Xi)\geq \#\Xi_0^x$ that the equality fails only when $\#\pi_x(\Xi)> \#\Xi_0^x$, i.e., there exists at least one point
$(\bar{x}, \bar{y})\in \Xi$ such that $\bar{x}\notin \pi_x(\Xi_0^x)$. However, $(\bar{x}, \bar{y})\in \Xi$ simply means that there exists some $0<j\leq \#\pi_y(\Xi)-1$ such that $(\bar{x}, \bar{y})\in \Xi_j^x$. Thus \eqref{e:pix} implies that $\bar{x}\in \pi_x(\Xi_j^x)\subseteq \pi_x(\Xi_0^x)$ which contradicts $\bar{x}\notin \pi_x(\Xi_0^x)$.

In the rest of the proof, we will use induction on $h$ to show that
\begin{equation}\label{e:nn}
    n_h=\#\Xi_h^y-1,\quad h=0, \ldots, \#\pi_x(\Xi)-1,
\end{equation}
which leads to $S_x(\Xi)=S_y(\Xi)$ immediately. When
$h=0$, for every $(\bar{x}, \bar{y})\in \Xi_{\#\pi_y(\Xi)-1}^x$, it follows from \eqref{e:pix} that $\bar{x}\in \pi_x(\Xi_{\#\pi_y(\Xi)-1}^x)\subset \pi_x(\Xi_j^x), 0\leq j\leq \#\pi_y(\Xi)-2$, which means that $\#\{\Xi\cap \{x=\bar{x}\}\}\geq \#\pi_y(\Xi)$. But $\#\{\Xi\cap \{x=\bar{x}\}\}\leq \#\pi_y(\Xi)$ is trivial, we have $\#\Xi_0^y=\#\{\Xi\cap \{x=\bar{x}\}\}=\#\pi_y(\Xi)=n_0+1$, namely \eqref{e:nn}
is true for $h=0$.

Now assume \eqref{e:nn} for $0\leq h \leq k< \#\pi_x(\Xi)-1$, i.e., $n_h=\#\Xi_h^y-1, h=0, \ldots, k$. It turns out that there exist distinct $\bar{x}_0, \bar{x}_1, \ldots, \bar{x}_k\in \pi_x(\Xi)$ such that $\#\{\Xi\cap\{x=\bar{x}_h\}\}=n_h+1, 0\leq h \leq k$. Thus, for every $0\leq h \leq k$, $n_h\geq n_{k+1}$ implies $\bar{x}_h\in \pi_x(\Xi_{n_{k+1}}^x)$. Since $(k+1, n_{k+1})\in S_x(\Xi)$, there exists at least one point
$(\bar{x}_{k+1}, \bar{y}_{k+1})\in \Xi_{n_{k+1}}^x$ that is not in $\Xi_j^y, j=0, \ldots, k$. By \eqref{e:pix}, a similar argument leads to $\#\{\Xi\cap \{x=\bar{x}_{k+1}\}\}\geq n_{k+1}+1$ which implies $\#\Xi_{k+1}^y\geq n_{k+1}+1$. On the other side, it follows from induction hypothesis that every point in $\Xi_{k+1}^y$ belongs to some $\Xi_j^x, 0\leq j\leq n_{k+1}$, which implies that $\#\Xi_{k+1}^y\leq n_{k+1}+1$, therefore we have $\#\Xi_{k+1}^y=n_{k+1}+1$, namely \eqref{e:nn} holds for $h=k+1$. Theorem \ref{t:lowercar} immediately implies that $\Xi$ is Cartesian.

Swapping the roles of $x$ and $y$, the other statement can be proved similarly.

\end{proof}

As mentioned in Section \ref{s:int}, \cite{Sau2004} provides the Gr\"{o}bner \'{e}scalier of the vanishing ideal of a Cartesian set in theory. In view of a later application,  we restate the result only in case $d=2$.

\begin{theorem}\textup{\cite{Sau2004}}\label{t:Sauer}
Let $\Xi\subset \mathbb{F}^2$ be an $\mathcal{A}$-Cartesian set. Then Gr\"{o}bner \'{e}scalier $\mathcal{N}_\prec(\mathcal {I}(\Xi))$
w.r.t. any monomial order $\prec$ is identical to
\begin{equation}\label{e:lowerbas}
\mathfrak{M}_\mathcal{A}:=\{x^iy^j: (i, j)\in \mathcal {A}\}.
\end{equation}
\end{theorem}

Theorem \ref{t:Sauer} indicates that an $\mathcal{A}$-Cartesian set in $\mathbb{F}^2$ has the advantage that the Gr\"{o}bner \'{e}scalier of its vanishing ideal can be easily obtained from the structure of $\mathcal{A}$. Nevertheless, Theorem \ref{t:lower} illustrates that the distribution of a Cartesian set in $\mathbb{F}^2$ is highly restricted. Naturally, we wonder if there exists another type of finite nonempty sets with ``freer" distribution and \eqref{e:lowerbas}-like property.

\begin{definition}\label{d:tower}
  Keep the notation above. A finite nonempty set $\Xi$ in $\mathbb{F}^2$ is termed $x$-\emph{tower} (resp. $y$-\emph{tower}) if $S_x(\Xi)$ is $x$-strict \textup{(}resp. $S_y(\Xi)$ is $y$-strict\textup{)} and there exist two injective functions $\mathrm{x, y}:\mathbb{N}_0\rightarrow \mathbb{F}$ as well as $\mathrm{b}_y(S_x(\Xi))+1$ \textup{(}resp. $\mathrm{b}_x(S_y(\Xi))+1$\textup{)} permutations $\sigma_0^x, \ldots, \sigma_{\mathrm{b}_y(S_x(\Xi))}^x$ \textup{(}resp. $\sigma_0^y, \ldots, \sigma_{\mathrm{b}_x(S_y(\Xi))}^y$\textup{)} of set $\{0, 1, \ldots, \mathrm{b}_x(S_x(\Xi))\}$ \textup{(}resp. $\{0, 1, \ldots, \mathrm{b}_y(S_y(\Xi))\}$\textup{)} such that $\Xi$ can be written as
\begin{align}
    \Xi :=& \{(\mathrm{x}(\sigma_j^x(i)),\mathrm{y}(j)): (i,j)\in S_x(\Xi) \}\label{e:towerx}\\
    \mbox{\textup{(}resp. }\Xi :=& \{(\mathrm{x}(i),\mathrm{y}(\sigma_j^y(j))): (i,j)\in S_y(\Xi) \}\textup{)}\label{e:towery}.
\end{align}
\end{definition}

Fix the injective functions $\mathrm{x}$ and $\mathrm{y}$. Comparing \eqref{e:towerx} with \eqref{e:newlowersetx}, we find that if the permutations in \eqref{e:towerx} are all identical, then \eqref{e:towerx} is same as \eqref{e:newlowersetx} in form. Assume that $(i, j)\in S_x(\Xi)$. If $\Xi\subset \mathbb{F}^2$ is Cartesian, by \eqref{e:newlowersetx}, the corresponding point of $(i, j)$ in $\Xi$ must be $(\mathrm{x}(i), \mathrm{y}(j))$. But when $\Xi$ is $x$-tower, since $\sigma_j^x$ is arbitrary, the corresponding point of $(i, j)$ could be any one of $(\mathrm{x}(h), \mathrm{y}(j)), h=0, \ldots, \mathrm{b}_x(S_x(\Xi))$. Symmetrically, a $y$-tower set has the same behavior in vertical direction. Then $\mathrm{b}_x(S_x(\Xi))=\#\Xi_0^x-1$ and $\mathrm{b}_y(S_y(\Xi))=\#\Xi_0^y-1$ lead to the following criterion for tower sets instantly.

\begin{theorem}\label{t:towerc}
A finite nonempty set $\Xi\subset \mathbb{F}^2$ with decompositions \eqref{e:decom} is
$x$-tower \textup{(}resp. $y$-tower\textup{)} if and only if $S_x(\Xi)$ is $x$-strict \textup{(}resp. $S_y(\Xi)$ is $y$-strict\textup{)} and
\begin{align*}
    \pi_x(\Xi_0^x) &\supsetneq \pi_x(\Xi_j^x),\quad j=1,2,\ldots, \#\pi_y(\Xi)-1\\
    (\mbox{resp. } \pi_y(\Xi_0^y) &\supsetneq \pi_y(\Xi_i^y),\quad i=1,2,\ldots, \#\pi_x(\Xi)-1).
\end{align*}
\end{theorem}

Subfigure (a) of Fig. \ref{f:2} illustrates an $x$-tower set $\Xi$ with lower set $S_x(\Xi) = (11, 8, 6, 3, 1, 0)$. It is easy to check that the conditions in Theorem \ref{t:towerc} are satisfied.

Comparing Theorem \ref{t:towerc} with Theorem \ref{t:lower}, we find that in horizontal (resp. vertical) direction the distribution of an $x$-tower (resp. $y$-tower) set is ``freer" than a Cartesian set. Nonetheless, when it comes to the number of the points on each line, Cartesian sets are winners this time, because their $S_x(\Xi)$($=S_y(\Xi)$) are not restricted to be $x$-strict or $y$-strict.

By Theorem \ref{t:lower}, a tower set $\Xi\subset \mathbb{F}^2$ becomes a Cartesian set if and only if \eqref{e:pix} or \eqref{e:piy} is satisfied.
Conversely, it follows from Theorem \ref{t:towerc} that an $\mathcal{A}$-Cartesian set $\Xi\subset \mathbb{F}^2$ is $x$-tower(resp. $y$-tower) if and only if $\mathcal{A}$ is $x$-strict (resp. $y$-strict). Consequently, it turns out that the notions of Cartesian set and tower set in $\mathbb{F}^2$ are not mutually exclusive. Nevertheless, Theorem \ref{t:lower} and \ref{t:towerc} also implies that most tower sets are not Cartesian and vice versa. For example, set $\Xi$ in (a) of Fig. \ref{f:2} is $x$-tower but not Cartesian while set $\Xi'$ in (b) of Fig. \ref{f:2} is Cartesian but not $x$-tower or $y$-tower.

\begin{figure}[!htbp]
\centering
\subfigure[$\Xi$]{\includegraphics[width=6cm,height=6cm]{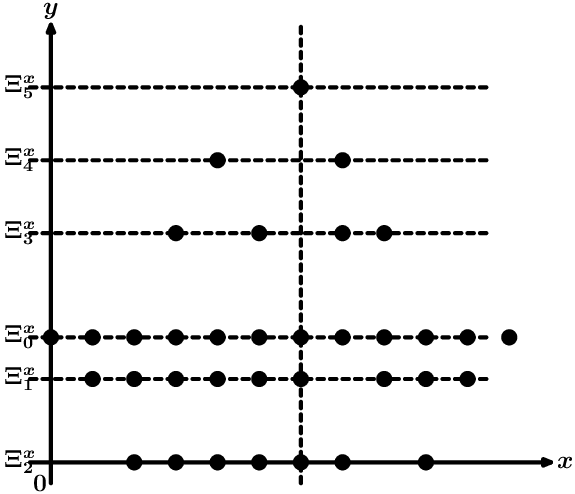}}
\subfigure[$\Xi'$]{\includegraphics[width=6cm,height=6cm]{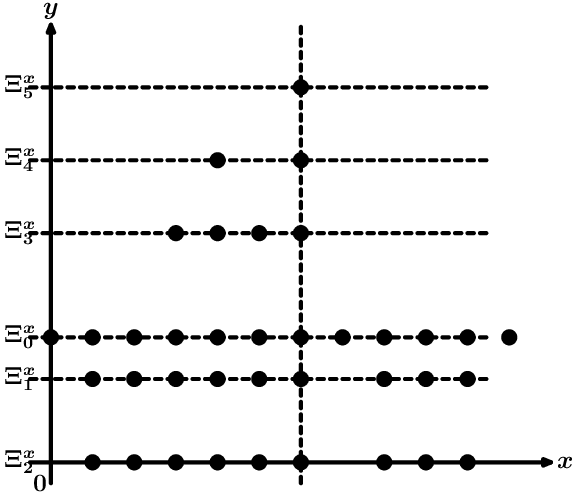}}

\caption{A non-Cartesian tower set and a non-tower Cartesian set}\label{f:2}  
\end{figure}

\section{Main Results}\label{s:result}


\subsection{Gr\"{o}bner \'{e}scalier}

We need the following lemma and definition before we give Theorem \ref{t:td} that theoretically provides the Gr\"{o}bner \'{e}scalier of the vanishing ideal of an $x$-tower(resp. $y$-tower) set in $\mathbb{F}^2$ w.r.t.
$\prec_{\mathrm{grlex}}$ (resp. $\prec_{\mathrm{grevlex}}$).

\begin{lemma}\textup{\cite{DZL2005}}\label{l:line}
  Let $\Xi=\{(x_0,y_0), (x_1, y_0), \ldots, (x_m,y_0)\}\subset \mathbb{F}^2$ be a set of
  $m+1$ distinct points on line $y=y_0$. Then
  $$
  \mathcal{I}(\Xi)=\langle (x-x_{0})(x-x_{1})\cdots(x-x_{m}),y-y_0\rangle.
  $$
\end{lemma}

\begin{definition}\textup{\cite{CLO2007}}
  Fix a monomial order $\prec$ and let $F = \{f_1,\ldots, f_s\}\subset \Pi^d$ with $f_i\neq 0$. Given
$f, f'\in \Pi^d$, we say that $f$ \emph{reduces} to $f'$ modulo $f_1$ in one step, written
$$
f \xrightarrow{f_1} f'
$$
if and only if $\mathrm{LT}_\prec(f_1)$ divides a nonzero term $\mathrm{t}$ that appears in $f$ and
$$
f'=f-\frac{\mathrm{t}}{\mathrm{LT}_\prec(f_1)}f_1.
$$
Moreover, we say that $f$ \emph{reduces} to $f'$ modulo $F$, denoted
$$
f \xrightarrow{F}_+ f',
$$
if and only if there exist a sequence of indices $i_1, \ldots, i_t\in \{1,\ldots, s\}$  and a
sequence of polynomials $h_1, \ldots, h_{t-1}\in \Pi^d$ such that
$$
f\xrightarrow{f_{i_1}} h_1 \xrightarrow{f_{i_2}} h_2 \xrightarrow{f_{i_3}} \cdots \xrightarrow{f_{i_{t-1}}} h_{t-1} \xrightarrow{f_{i_t}} f'.
$$
\end{definition}

\begin{theorem}\label{t:td}
Given an $x$-tower\textup{(}resp. $y$-tower\textup{)} set $\Xi\subset \mathbb{F}^2$. The Gr\"{o}bner \'{e}scalier of vanishing ideal $\mathcal{I}(\Xi)$ w.r.t. $\prec_{\mathrm{grlex}}$ \textup{(}resp. $\prec_{\mathrm{grevlex}}$\textup{)} is $\mathfrak{M}_{S_x(\Xi)}$ \textup{(}resp. $\mathfrak{M}_{S_y(\Xi)}$\textup{)}.
\end{theorem}

\begin{proof}
We only offer the proof of the first statement. The second one can be verified in very like
fashion.

Retain all the notation established previously. Fix monomial order as $\prec_{\mathrm{grlex}}$, and for simplicity the symbol will be omitted in the rest of the proof where no confusion arises. Suppose that $x$-tower set $\Xi$ has the decompositions \eqref{e:decom}. For convenience, we set $m_j:=\#\Xi_j^x-1, j=0, \ldots, \nu$, where $\nu=\#\pi_y(\Xi)-1=\mathrm{b}_y(S_x(\Xi))$. Thus, \eqref{e:sxsy} implies $S_x(\Xi)=\mathrm{L}_x(m_0, m_1, \ldots, m_\nu)$. Fix $0\leq j\leq \nu$. By Lemma \ref{l:line}, the
ideal
\begin{align*}
    \mathcal{I}_j:=\mathcal{I}(\Xi_j^x)=&\mathcal{I}(\{(\mathrm{x}(\sigma_j^x(i)),\mathrm{y}(j)): (i,j)\in S_x(\Xi) \})\\
    =&\left\langle \prod_{(i, j)\in S_x(\Xi)} (x-\mathrm{x}(\sigma_j^x(i))), y-\mathrm{y}(j)\right\rangle.
\end{align*}
Obviously, $\mathcal{I}_j$'s are pairwise comaximal. Hence
$$
\mathcal{I}(\Xi)=\bigcap_{j=0}^\nu \mathcal{I}_j = \prod_{j=0}^\nu \mathcal{I}_j=\prod_{j=0}^\nu \left\langle \prod_{(i, j)\in S_x(\Xi)} (x-\mathrm{x}(\sigma_j^x(i))), y-\mathrm{y}(j)\right\rangle.
$$
Let $G_N$ be the reduced
Gr\"{o}bner basis for ideal $\prod_{j=0}^N \mathcal{I}_j$ w.r.t. $\prec_{\mathrm{grlex}}$,
$0\leq N \leq \nu$. We will use induction on $N$ to prove
\begin{equation}\label{e:9}
\mathrm{LM}(G_N) = \{x^{m_0+1},x^{m_1+1}y, \ldots, x^{m_N+1}y^N,y^{N+1}\}, \quad N=0, 1, \ldots, \nu.
\end{equation}
First of all,
$$
\mathrm{LM}\left(\prod_{(i, 0)\in S_x(\Xi)} (x-\mathrm{x}(\sigma_0^x(i)))\right)=x^{\#\Xi_0^x}=x^{m_0+1}
$$
 leads to (\ref{e:9}) immediately for $N=0$.

It follows from Theorem \ref{t:towerc} that $m_0 > m_1$ and $\pi_x(\Xi_1^x)\subsetneq
\pi_x(\Xi_0^x)$. Therefore, we have
\begin{align*}
&\mathcal{I}_0 \cdot \mathcal{I}_1\\
 = & \left\langle \prod_{(i, 0)\in S_x(\Xi)} (x-\mathrm{x}(\sigma_0^x(i))), y-\mathrm{y}(0)\right\rangle
\cdot
\left\langle \prod_{(i, 1)\in S_x(\Xi)} (x-\mathrm{x}(\sigma_1^x(i))), y-\mathrm{y}(1)\right\rangle\\
= & \Bigg\langle \prod_{(i, 0)\in S_x(\Xi)} (x-\mathrm{x}(\sigma_0^x(i))), \Bigg(\prod_{(i, 0)\in S_x(\Xi)} (x-\mathrm{x}(\sigma_0^x(i)))\Bigg) (y-\mathrm{y}(1)),\\
& \phantom{\Bigg\langle}\Bigg(\prod_{(i, 1)\in S_x(\Xi)} (x-\mathrm{x}(\sigma_1^x(i)))\Bigg)(y-\mathrm{y}(0)), \prod_{j=0}^1(y-\mathrm{y}(j))\Bigg\rangle\\
= & \Bigg\langle \prod_{(i, 0)\in S_x(\Xi)} (x-\mathrm{x}(\sigma_0^x(i))), \Bigg(\prod_{(i, 1)\in S_x(\Xi)} (x-\mathrm{x}(\sigma_1^x(i)))\Bigg)(y-\mathrm{y}(0)), \prod_{j=0}^1(y-\mathrm{y}(j))\Bigg\rangle
\end{align*}
where the last equality holds by
$$
\prod_{(i, 0)\in S_x(\Xi)} (x-\mathrm{x}(\sigma_0^x(i)))\Bigg|\Bigg(\prod_{(i, 0)\in S_x(\Xi)} (x-\mathrm{x}(\sigma_0^x(i)))\Bigg) (y-\mathrm{y}(1)).
$$
Then $\mathrm{LM}(G_1) = \{x^{m_0+1},x^{m_1+1}y,y^2\}$ follows which means
that \eqref{e:9} holds for $N=1$.

Similarly, by $m_1>m_2$ and $\pi_x(\Xi_2^x)\subsetneq
\pi_x(\Xi_0^x)$, we obtain, after some easy
computations,
\begin{align*}
&\mathcal{I}_0 \cdot \mathcal{I}_1 \cdot \mathcal{I}_2\\
=& \Bigg\langle \prod_{(i, 0)\in S_x(\Xi)} (x-\mathrm{x}(\sigma_0^x(i))), \prod_{j=0}^2(y-\mathrm{y}(j)), \\
& \phantom{\Bigg\langle}\Bigg(\prod_{(i, 1)\in S_x(\Xi)} (x-\mathrm{x}(\sigma_1^x(i)))\Bigg)\Bigg(\prod_{(i, 2)\in S_x(\Xi)} (x-\mathrm{x}(\sigma_2^x(i)))\Bigg)(y-\mathrm{y}(0)),\\
&\phantom{\Bigg\langle}\Bigg(\prod_{(i, 1)\in S_x(\Xi)} (x-\mathrm{x}(\sigma_1^x(i)))\Bigg)(y-\mathrm{y}(0))(y-\mathrm{y}(2)),\\
&\phantom{\Bigg\langle}\Bigg(\prod_{(i, 2)\in S_x(\Xi)} (x-\mathrm{x}(\sigma_2^x(i)))\Bigg)\prod_{j=0}^1(y-\mathrm{y}(j))\Bigg\rangle\\
=: & \left\langle g_0^{(2)}, g_1^{(2)}, g_2^{(2)}, g_3^{(2)}, g_4^{(2)} \right\rangle.
\end{align*}
We let $\hat{q}, \hat{r}\in \Pi^1$ be the quotient and
the remainder of the division of $\prod_{(i, 1)\in S_x(\Xi)} (x-\mathrm{x}(\sigma_1^x(i)))$ by
$\prod_{(i, 2)\in S_x(\Xi)} (x-\mathrm{x}(\sigma_2^x(i)))$ respectively, namely
$$
\prod_{(i, 1)\in S_x(\Xi)} (x-\mathrm{x}(\sigma_1^x(i)))=\hat{q} \Bigg(\prod_{(i, 2)\in S_x(\Xi)} (x-\mathrm{x}(\sigma_2^x(i)))\Bigg) + \hat{r}.
$$
Denote the remainder of $g_3^{(2)}$ w.r.t. $g_4^{(2)}$ by $\bar{g}_3^{(2)}$. One can
check readily that
$$
\bar{g}_3^{(2)}=g_3^{(2)}-\hat{q} g_4^{(2)}.
$$
On the other hand,
\begin{align*}
    g_2^{(2)}=\frac{1}{\mathrm{y}(1)-\mathrm{y}(2)}\Bigg[&\Bigg(\prod_{(i, 2)\in S_x(\Xi)} (x-\mathrm{x}(\sigma_2^x(i)))\Bigg)g_3^{(2)}-\\
    &\Bigg(\prod_{(i, 1)\in S_x(\Xi)} (x-\mathrm{x}(\sigma_1^x(i)))\Bigg)g_4^{(2)}\Bigg]
\end{align*}
implies $g_2^{(2)} \xrightarrow{\{g_3^{(2)}, g_4^{(2)}\}}_+ 0 $. Consequently,
\begin{align*}
\mathcal{I}_0 \cdot \mathcal{I}_1 \cdot \mathcal{I}_2= & \left\langle\ g_0^{(2)}, g_1^{(2)},
\bar{g}_3^{(2)}, g_4^{(2)}\right\rangle.
\end{align*}
We claim that
$$
G_2=G'_2:=\left\{g_0^{(2)}, g_1^{(2)}, (\mathrm{y}(1)-\mathrm{y}(2))^{-1}\bar{g}_3^{(2)}, g_4^{(2)}\right\},
$$
where $\mathrm{y}(1)-\mathrm{y}(2)=\mathrm{LC}\left(\bar{g}_3^{(2)}\right)$, i.e., $(\mathrm{y}(1)-\mathrm{y}(2))^{-1}\bar{g}_3^{(2)}$ is
monic.

In fact, if $S(f, g)$ stands for the S-polynomial of polynomials $f, g\in
\Pi^2$, then $S\left(g_0^{(2)}, g_1^{(2)}\right)\xrightarrow{G'_2}_+ 0$ follows immediately because
$\mathrm{LM}\left(g_0^{(2)}\right)$ and $\mathrm{LM}\left(g_1^{(2)}\right)$ are relatively prime. Observing that
$\prod_{(i, 2)\in S_x(\Xi)} (x-\mathrm{x}(\sigma_2^x(i)))$ is a factor of $\prod_{(i, 0)\in S_x(\Xi)} (x-\mathrm{x}(\sigma_0^x(i)))$, we get
\begin{align*}
    &S\left(g_0^{(2)}, g_4^{(2)}\right)\\
    =&y^2g_0^{(2)}-x^{m_0-m_2}g_4^{(2)}\\
    =&y^2g_0^{(2)}-x^{m_0-m_2}g_4^{(2)}-\Bigg(\prod_{j=0}^1(y-\mathrm{y}(j))\Bigg)g_0^{(2)}+
     \Bigg(\prod_{j=0}^1(y-\mathrm{y}(j))\Bigg)g_0^{(2)}\\
    =&y^2g_0^{(2)}-x^{m_0-m_2}g_4^{(2)}-\Bigg(\prod_{j=0}^1(y-\mathrm{y}(j))\Bigg)g_0^{(2)}+
     \frac{\prod_{(i, 0)\in S_x(\Xi)} (x-\mathrm{x}(\sigma_0^x(i)))}
{\prod_{(i, 2)\in S_x(\Xi)} (x-\mathrm{x}(\sigma_2^x(i)))}g_4^{(2)}\\
    =&((\mathrm{y}(0)+\mathrm{y}(1))y-\mathrm{y}(0)\mathrm{y}(1))g_0^{(2)}+
    \left(\frac{\prod_{(i, 0)\in S_x(\Xi)} (x-\mathrm{x}(\sigma_0^x(i)))}
{\prod_{(i, 2)\in S_x(\Xi)} (x-\mathrm{x}(\sigma_2^x(i)))}-x^{m_0-m_2}\right)g_4^{(2)}.
\end{align*}
It is easy to see that $\mathrm{LM}\left(S\left(g_0^{(2)}, g_4^{(2)}\right)\right)=x^{m_0+1}y$. Moreover, a simple computation leads to:
\begin{align*}
    \mathrm{LM}\left(((\mathrm{y}(0)+\mathrm{y}(1))y-\mathrm{y}(0)\mathrm{y}(1))g_0^{(2)}\right)&=x^{m_0+1}y,\\
    \mathrm{LM}\left(\left(\frac{\prod_{(i, 0)\in S_x(\Xi)} (x-\mathrm{x}(\sigma_0^x(i)))}
{\prod_{(i, 2)\in S_x(\Xi)} (x-\mathrm{x}(\sigma_2^x(i)))}-x^{m_0-m_2}\right)g_4^{(2)}\right)&=x^{m_0}y^2,
\end{align*}
therefore
\begin{align*}
& \max_{\prec_{\mathrm{grlex}}}\Bigg(
\mathrm{LM}\left(((\mathrm{y}(0)+\mathrm{y}(1))y-\mathrm{y}(0)\mathrm{y}(1))g_0^{(2)}\right), \\
& \phantom{\max_{\prec_{\mathrm{grlex}}}\Bigg(}\mathrm{LM}\left(\left(\frac{\prod_{(i, 0)\in S_x(\Xi)} (x-\mathrm{x}(\sigma_0^x(i)))}
{\prod_{(i, 2)\in S_x(\Xi)} (x-\mathrm{x}(\sigma_2^x(i)))}-x^{m_0-m_2}\right)g_4^{(2)}\right)\Bigg)\\
=& x^{m_0+1}y=\mathrm{LM}\left(S\left(g_0^{(2)}, g_4^{(2)}\right)\right)
\end{align*}
implies $S\left(g_0^{(2)}, g_4^{(2)}\right)\xrightarrow{G_2'}_+ 0$. Similarly,
$$
    S\left(g_1^{(2)}, g_4^{(2)}\right)=\Bigg(x^{m_2+1}-\prod_{(i, 2)\in S_x(\Xi)} (x-\mathrm{x}(\sigma_2^x(i)))\Bigg)g_1^{(2)}-\mathrm{y}(2)g_4^{(2)}
$$
and $\mathrm{LM}\left(S\left(g_1^{(2)}, g_4^{(2)}\right)\right)=x^{m_2+1}y^2$ imply
\begin{align*}
&\max_{\prec_{\mathrm{grlex}}}\Bigg(\mathrm{LM}\Bigg(\Bigg(x^{m_2+1}-\prod_{(i, 2)\in S_x(\Xi)} (x-\mathrm{x}(\sigma_2^x(i)))\Bigg)g_1^{(2)}\Bigg),
                 \mathrm{LM}\left(-\mathrm{y}(2)g_4^{(2)}\right)\Bigg)\\
               =&\max_{\prec_{\mathrm{grlex}}}\left(x^{m_2}y^3, x^{m_2+1}y^2\right) = x^{m_2+1}y^2 = \mathrm{LM}\left(S\left(g_1^{(2)},
               g_4^{(2)}\right)\right),
\end{align*}
which means that $S\left(g_1^{(2)}, g_4^{(2)}\right)\xrightarrow{G_2'}_+ 0$.

In like manner, we can also show that
$$
S\left(g_0^{(2)}, (\mathrm{y}(1)-\mathrm{y}(2))^{-1}\bar{g}_3^{(2)}\right)\xrightarrow{G_2'}_+ 0,\quad S\left(g_1^{(2)},
(\mathrm{y}(1)-\mathrm{y}(2))^{-1}\bar{g}_3^{(2)}\right)\xrightarrow{G_2'}_+ 0.
$$
Hence, there only remains $S\left((\mathrm{y}(1)-\mathrm{y}(2))^{-1}\bar{g}_3^{(2)}, g_4^{(2)}\right)
\xrightarrow{G_2'}_+ 0$ to be verified. Actually, it can be deduced that
\begin{align*}
    &S\left((\mathrm{y}(1)-\mathrm{y}(2))^{-1}\bar{g}_3^{(2)}, g_4^{(2)}\right)\\
    =&\frac{\hat{r}}{\mathrm{y}(1)-\mathrm{y}(2)}g_1^{(2)}+\mathrm{y}(1)\cdot (\mathrm{y}(1)-\mathrm{y}(2))^{-1} \bar{g}_3^{(2)}+
\left(\hat{q}-x^{m_1-m_2}\right)g_4^{(2)}
\end{align*}
whose leading monomial is $x^{m_1+1}y$. Thus
\begin{align*}
    \mathrm{LM}\left(\frac{\hat{r}}{\mathrm{y}(1)-\mathrm{y}(2)}g_1^{(2)}\right)&=x^{m_2}y^3,\\
    \mathrm{LM}\left(\mathrm{y}(1)\cdot (\mathrm{y}(1)-\mathrm{y}(2))^{-1} \bar{g}_3^{(2)}\right)&=x^{m_1+1}y,
\end{align*}
and
$$
    \mathrm{LM}\left(\left(\hat{q}-x^{m_1-m_2}\right)g_4^{(2)}\right)=x^{m_1}y^2
$$
imply that $S\left((\mathrm{y}(1)-\mathrm{y}(2))^{-1}\bar{g}_3^{(2)}, g_4^{(2)}\right) \xrightarrow{G_2'}_+ 0$.

Now, these arguments lead to the conclusion that S-polynomial $S(g, g')\xrightarrow{G'_2}_+ 0$
for any two distinct $g, g' \in G'_2$. By Buchberger's S-pair criterion,
$G_2'$ is a Gr\"{o}bner basis for $\mathcal{I}_0 \cdot \mathcal{I}_1 \cdot \mathcal{I}_2$
w.r.t. $\prec_{\mathrm{grlex}}$. Moreover, for every $g\in G_2'$, it is evident that
\begin{enumerate}
    \item $\mathrm{LC}(g)=1$,
    \item No monomial of $g$ lies in $\langle
\mathrm{LT}(G_2'-\{g\})\rangle$,
\end{enumerate}
which means that $G_2'$ is reduced, namely \eqref{e:9} holds for $N=2$.

Now, assume (\ref{e:9}) for $N=k, 3\leq k<\nu$.
Without loss of generality, we suppose that $G_k = \left\{g_0^{(k)}, \ldots, g_{k+1}^{(k)}\right\}$ with
\begin{align*}
\mathrm{LM}\left(g_i^{(k)}\right)=&x^{m_i+1}y^i,\quad i=0,\ldots,k , \\
\mathrm{LM}\left(g_{k+1}^{(k)}\right)=& y^{k+1},
\end{align*}
which imply that
\begin{align*}
g_0^{(k)}&=\prod_{(i, 0)\in S_x(\Xi)} (x-\mathrm{x}(\sigma_0^x(i))),\\
g_{k+1}^{(k)}&=\prod_{j=0}^k(y-\mathrm{y}(j)).
\end{align*}
When $N=k+1$, since $\Xi$ is $x$-tower, we obtain
\begin{align*}
\prod_{j=0}^{k+1}\mathcal{I}_j=& (\prod_{j=0}^{k}\mathcal{I}_j)
\mathcal{I}_{k+1}\\
 =&\Big\langle g_0^{(k)}, \ldots, g_{k+1}^{(k)} \Big\rangle
\cdot \Bigg\langle \prod_{(i, k+1)\in S_x(\Xi)} (x-\mathrm{x}(\sigma_{k+1}^x(i))), y-\mathrm{y}(k+1) \Bigg\rangle \\
 = & \Bigg\langle g_0^{(k+1)}, g_0^{(k)}(y-\mathrm{y}(k+1)),\\
  & \phantom{\Big\langle} g_1^{(k)}\Bigg(\prod_{(i, k+1)\in S_x(\Xi)} (x-\mathrm{x}(\sigma_{k+1}^x(i)))\Bigg), g_1^{(k)}(y-\mathrm{y}(k+1)),\\
  &          \cdots \\
  & \phantom{\Big\langle} g_{k-1}^{(k)}\Bigg(\prod_{(i, k+1)\in S_x(\Xi)} (x-\mathrm{x}(\sigma_{k+1}^x(i)))\Bigg), g_{k-1}^{(k)}(y-\mathrm{y}(k+1)),\\
  & \phantom{\Big\langle} g_k^{(k)}\Bigg(\prod_{(i, k+1)\in S_x(\Xi)} (x-\mathrm{x}(\sigma_{k+1}^x(i)))\Bigg), g_k^{(k)}(y-\mathrm{y}(k+1)),\\
  & \phantom{\Big\langle} g_{k+1}^{(k)}\Bigg(\prod_{(i, k+1)\in S_x(\Xi)} (x-\mathrm{x}(\sigma_{k+1}^x(i)))\Bigg), g_{k+1}^{(k)}(y-\mathrm{y}(k+1)) \Bigg\rangle,
\end{align*}
where $g_0^{(k+1)}=g_0^{(k)}$ is obvious hence $g_0^{(k)}(y-\mathrm{y}(k+1))$ can be removed. By the induction hypothesis, we have
$$
    g_{k+2}^{(k+1)}:=g_{k+1}^{(k)}(y-\mathrm{y}(k+1))=\prod_{j=0}^{k+1}(y-\mathrm{y}(j)).
$$
We denote polynomial $g_{k+1}^{(k)}(\prod_{(i, k+1)\in S_x(\Xi)} (x-\mathrm{x}(\sigma_{k+1}^x(i))))$ by
$g_{k+1}^{(k+1)}$. It follows from the induction hypothesis that $\mathrm{LM}\left(g_{k+1}^{(k+1)}\right)=x^{m_{k+1}+1}y^{k+1}$. Set $E_1:=\left\{g_{k+1}^{(k+1)}\right\}$. Suppose that
$g_k^{(k)}(y-\mathrm{y}(k+1))\xrightarrow{E_1}_+ g_k^{(k+1)}$. Since $m_k > m_{k+1}$, we have
$$
\mathrm{LM}\left(g_k^{(k+1)}\right)=x^{m_k+1}y^k.
$$
Recall case $N=2$. It is easy to see that
$$
g_k^{(k)}\Bigg(\prod_{(i, k+1)\in S_x(\Xi)} (x-\mathrm{x}(\sigma_{k+1}^x(i)))\Bigg) \xrightarrow{F_1}_+ 0
$$
where $F_1:=\left\{g_{k+1}^{(k+1)}, g_k^{(k+1)}\right\}$, which means that $g_k^{(k)}(\prod_{(i, k+1)\in S_x(\Xi)} (x-\mathrm{x}(\sigma_{k+1}^x(i))))$ can be
removed from the original ideal basis.

Set $E_2:=F_1$,  and suppose that $g_{k-1}^{(k)}(y-\mathrm{y}(k+1)) \xrightarrow{E_2}_+
g_{k-1}^{(k+1)}$. We similarly deduce that
$\mathrm{LM}\left(g_{k-1}^{(k+1)}\right)=x^{m_{k-1}+1}y^{k-1}$ and
$$
g_{k-1}^{(k)}\Bigg(\prod_{(i, k+1)\in S_x(\Xi)} (x-\mathrm{x}(\sigma_{k+1}^x(i)))\Bigg) \xrightarrow{F_2}_+ 0
$$
where $F_2:=\left\{g_{k+1}^{(k+1)}, g_{k-1}^{(k+1)}\right\}$.

In this way, we construct two sequences $(E_1, E_2,\ldots, E_k)$ and $(F_1, F_2, \ldots,
F_k)$ such that
\begin{align*}
g_{i}^{(k)}(y-\mathrm{y}(k+1)) \xrightarrow{E_{k+1-i}}_+&
g_{i}^{(k+1)},\\
g_{i}^{(k)}\Bigg(\prod_{(i, k+1)\in S_x(\Xi)} (x-\mathrm{x}(\sigma_{k+1}^x(i)))\Bigg) \xrightarrow{F_{k+1-i}}_+& 0,
\end{align*}
where, for $i=1,\ldots,k$,
\begin{alignat*}{2}
E_i&=\left\{g_{k+2-i}^{(k+1)}, \ldots, g_{k+1}^{(k+1)}\right\},\\
F_i&=\left\{g_{k+1}^{(k+1)}, g_{k+1-i}^{(k+1)}\right\},\\
\mathrm{LM}\left(g_i^{(k+1)}\right)&=x^{m_i+1}y^i.
\end{alignat*}
Let $\bar{g}_i^{(k+1)}=\mathrm{LC}\left(g_i^{(k+1)}\right)^{-1} g_i^{(k+1)}, i=0, \ldots, k+2$. With similar
arguments used in $N=2$ case, we can finally prove that
$$
G_{k+1}=\left\{\bar{g}_0^{(k+1)}, \bar{g}_1^{(k+1)},\ldots,\bar{g}_{k+2}^{(k+1)}\right\},
$$
with $\mathrm{LM}\left(\bar{g}_i^{(k+1)}\right)=x^{m_i+1}y^i, i=0, \ldots, k+1$, and $\mathrm{LM}\left(\bar{g}_{k+2}^{(k+1)}\right)=y^{k+2}$,
namely \eqref{e:9} holds for $N=k+1$. Consequently, we have
$$
\mathrm{LM}(G_\nu) = \{x^{m_0+1},x^{m_1+1}y, \ldots, x^{m_\nu+1}y^\nu,y^{\nu+1}\}.
$$
Recalling \eqref{e:nf} and \eqref{e:sxsy}, we have
$$
\mathcal{N}(\mathcal{I}(\Xi))= \{x^iy^j: (i, j)\in
S_x(\Xi)\} = \mathfrak{M}_{S_x(\Xi)},
$$
which complete the proof.
\end{proof}

\subsection{Newton Basis}

The $\prec\mspace{-8mu}-$\emph{degree} of a nonzero polynomial $f\in \Pi^2$ (see \cite{dBo2007}), denoted by $\delta_\prec(f)$,
was defined to be $(i,j)\in \mathbb{N}_0^2$ satisfying
$$
x^iy^j=\mathrm{LM}_\prec(f).
$$
For every pair of polynomials $f, g\in \Pi^2$, if $\delta(f) \prec \delta(g)$ then
we say that $f$ is of \emph{lower $\prec\mspace{-8mu}-$degree} than $g$ and use the abbreviation
$$
f\prec g:=\delta(f) \prec \delta(g).
$$
In addition, $f\preceq g$ is interpreted as the $\prec\mspace{-8mu}-$degree of $f$ is lower than or equal to
that of $g$.

Given a finite nonempty set $\Xi=\{\xi^{(0)}, \xi^{(1)}, \ldots, \xi^{(\mu-1)}\}\subset \mathbb{F}^2$. For fixed monomial order $\prec$, the Gr\"{o}bner \'{e}scalier $\mathcal{N}_\prec(\mathcal{I}(\Xi))$ trivially forms the \emph{monomial basis} for $\mu$-dimensional $\mathbb{F}$-linear space $\mathcal{P}_\prec(\Xi):=\mathrm{Span}_\mathbb{F}\mathcal{N}_\prec(\mathcal{I}(\Xi))\subset \Pi^2$ that complements $\mathcal{I}(\Xi)$, i.e.
$$
\Pi^2=\mathcal{I}(\Xi)\oplus \mathcal{P}_\prec(\Xi).
$$
Moreover, if subset $\{p_0, p_1, \ldots, p_{\mu-1}\}\subset \mathcal{P}_\prec(\Xi)$, with $p_0 \prec p_1\prec \cdots \prec p_{\mu-1}$, satisfying
$$
p_j(\xi^{(i)})=\delta_{ij},\quad 0\leq i \leq j \leq \mu-1,
$$
then $\{p_0, p_1, \ldots, p_{\mu-1}\}$ is called a \emph{Newton interpolation basis} for $\mathcal{P}_\prec(\Xi)$.

Consequently, from Theorem \ref{t:td}, $\Xi$ is $x$-tower (resp. $y$-tower) implies that $\mathcal{P}_{\prec_{\mathrm{grlex}}}(\Xi)=\mathrm{Span}_\mathbb{F}\mathfrak{M}_{S_x(\Xi)}$(resp.  $\mathcal{P}_{\prec_{\mathrm{grevlex}}}(\Xi)=\mathrm{Span}_\mathbb{F}\mathfrak{M}_{S_y(\Xi)}$).
The next two theorems present Newton bases for $\mathcal{P}_{\prec_{\mathrm{grlex}}}(\Xi)$ and $\mathcal{P}_{\prec_{\mathrm{grevlex}}}(\Xi)$ respectively.

\begin{theorem}\label{t:phi^x}
Given an $x$-tower set $\Xi\subset \mathbb{F}^2$ that is expressed as \eqref{e:towerx}. Set polynomial
$$
\phi_{ij}^x:=c_{ij}^x \prod_{t=0}^{j-1}(y-\mathrm{y}(t))\prod_{s=0}^{i-1}(x-\mathrm{x}(\sigma_j^x(s))), \quad
(i,j)\in S_x(\Xi),
$$
where
$$
c_{ij}^x=\frac{1}{\prod_{t=0}^{j-1}(\mathrm{y}(j)-\mathrm{y}(t))\prod_{s=0}^{i-1}(\mathrm{x}(\sigma_j^x(i))
-\mathrm{x}(\sigma_j^x(s)))}\in
\mathbb{F}
$$ and the empty products are taken as 1.
Then for $(i,j), (m, n)\in S_x(\Xi)$ with $(i,j)\succeq_{\mathrm{inlex}} (m,n)$, we have
$$
\phi_{ij}^x\left(\mathrm{x}(\sigma_n^x(m)), \mathrm{y}(n)\right)=\delta_{(i,j), (m,n)},
$$
namely \begin{equation}\label{e:Qx}
\mathfrak{N}_{S_x(\Xi)}:=\{\phi_{ij}^x: (i, j)\in S_x(\Xi)\}
\end{equation}
forms a Newton interpolation basis for $\mathcal{P}_{\prec_{\mathrm{grlex}}}(\Xi)$.
\end{theorem}

\begin{proof}
Fix $(i, j)\in S_x(\Xi)$. If $(i, j)=(m, n)$, then
$$
\phi_{ij}^x(\mathrm{x}(\sigma_n^x(m)), \mathrm{y}(n))=c_{ij}^x\prod_{t=0}^{j-1}(\mathrm{y}(j)-\mathrm{y}(t))\prod_{s=0}^{i-1}(\mathrm{x}(\sigma_j^x(i))
-\mathrm{x}(\sigma_j^x(s)))\\
                     =c_{ij}^x/c_{ij}^x=1.
$$
Otherwise, $(i, j)\succneqq_{\mathrm{inlex}} (m, n)$ implies $j>n$ or $j=n$ and $i>m$. When
$j>n$, i.e., $j-1\geq n$,  we have
$$
\phi_{ij}^x\left(\mathrm{x}(\sigma_n^x(m)), \mathrm{y}(n)\right)=c_{ij}^x\prod_{t=0}^{j-1}(\mathrm{y}(n)-\mathrm{y}(t))\prod_{s=0}^{i-1}(\mathrm{x}(\sigma_n^x(m))
-\mathrm{x}(\sigma_j^x(s)))=0.
$$
If $j=n, i>m$, namely $i-1\geq m$, then
\begin{align*}
\phi_{ij}^x\left(\mathrm{x}(\sigma_n^x(m)), \mathrm{y}(n)\right)&=c_{ij}^x\prod_{t=0}^{j-1}(\mathrm{y}(n)-\mathrm{y}(t))\prod_{s=0}^{i-1}(\mathrm{x}(\sigma_n^x(m))
-\mathrm{x}(\sigma_j^x(s)))\\
&=c_{ij}^x\prod_{t=0}^{j-1}(\mathrm{y}(n)-\mathrm{y}(t))\prod_{s=0}^{i-1}(\mathrm{x}(\sigma_j^x(m))
-\mathrm{x}(\sigma_j^x(s)))\\
&=0,
\end{align*}
which leads to
$$
\phi_{ij}^x\left(\mathrm{x}(\sigma_n^x(m)), \mathrm{y}(n)\right)=0, \quad (i,j)\succ_{\mathrm{inlex}} (m,n),
$$
as desired. It is easy to check that
$\mathrm{Span}_\mathbb{F}\mathfrak{N}_x=\mathrm{Span}_\mathbb{F}\mathfrak{M}_{S_x(\Xi)}=\mathcal{P}_{\prec_{\mathrm{grlex}}}(\Xi)$.
\end{proof}

Similarly, we can prove the following theorem:

\begin{theorem}\label{t:phi^y}
Let $\Xi\subset \mathbb{F}^2$ be a $y$-tower set that is defined by \eqref{e:towery}. We let polynomial
$$
\phi_{ij}^y:=c_{ij}^y\prod_{s=0}^{i-1}(x-\mathrm{x}(s))\prod_{t=0}^{j-1}(y-\mathrm{y}(\sigma_i^y(t))),\quad
(i,j)\in S_y(\Xi),
$$
where
$$
c_{ij}^y=\frac{1}{\prod_{s=0}^{i-1}(\mathrm{x}(i)-\mathrm{x}(s))\prod_{t=0}^{j-1}(\mathrm{y}(\sigma_i^y(j))-
\mathrm{y}(\sigma_i^y(t)))}\in \mathbb{F}
$$
and the empty products are taken as 1. Then,
\begin{equation}\label{e:Qy}
\mathfrak{N}_{S_y(\Xi)}:=\{\phi_{ij}^y: (i, j)\in S_y(\Xi)\}
\end{equation}
is a Newton interpolation basis for $\mathcal{P}_{\prec_{\mathrm{grevlex}}}(\Xi)$ satisfying
$$
\phi_{ij}^y((\mathrm{x}(m), \mathrm{y}(\sigma_m^y(n))))=\delta_{(i,j),(m,n)}, \quad (i,j)\succeq_{\mathrm{lex}} (m,n).
$$
\end{theorem}

Now, we turn to $\prec_{\mathrm{lex}}$ and $\prec_{\mathrm{inlex}}$ cases. For every
finite nonempty set $\Xi\subset \mathbb{F}^2$, \cite{WZD2010} presents monomial and Newton interpolation bases for $\mathcal{P}_{\prec_{\mathrm{lex}}}(\Xi)$ and $\mathcal{P}_{\prec_{\mathrm{inlex}}}(\Xi)$. In the following, we restate the results with $\Xi$ limited to tower sets only.

\begin{lemma}\label{l:WZD}
Let $\Xi$ be an $x$-tower \textup{(}resp. $y$-tower\textup{)} set which is defined by \eqref{e:towerx} \textup{(}resp. \eqref{e:towery}\textup{)}. Then $\mathfrak{M}_{S_x(\Xi)}$ \textup{(}resp. $\mathfrak{M}_{S_y(\Xi)}$\textup{)} is the monomial basis and $\mathfrak{N}_{S_x(\Xi)}$ \textup{(}resp. $\mathfrak{N}_{S_y(\Xi)}$\textup{)} a Newton basis for $\mathcal{P}_{\prec_{\mathrm{lex}}}(\Xi)$ \textup{(}resp. $\mathcal{P}_{\prec_{\mathrm{inlex}}}(\Xi)$\textup{)}.
\end{lemma}

Combining Theorem \ref{t:td}--\ref{t:phi^y} and Lemma \ref{l:WZD}, we arrive at the following main theorem:

\begin{theorem}\label{t:final}
Let $\Xi$ be an $x$-tower \textup{(}resp. $y$-tower\textup{)} set that is defined by \eqref{e:towerx} \textup{(}resp. \eqref{e:towery}\textup{)}. Then $\mathfrak{M}_{S_x(\Xi)}$ \textup{(}resp. $\mathfrak{M}_{S_y(\Xi)}$\textup{)} is the monomial basis and $\mathfrak{N}_{S_x(\Xi)}$ \textup{(}resp. $\mathfrak{N}_{S_y(\Xi)}$\textup{)} a Newton interpolation basis for $\mathcal{P}_{\prec_{\mathrm{grlex}}}(\Xi)$ and $\mathcal{P}_{\prec_{\mathrm{lex}}}(\Xi)$ \textup{(}resp. $\mathcal{P}_{\prec_{\mathrm{grevlex}}}(\Xi)$ and $\mathcal{P}_{\prec_{\mathrm{inlex}}}(\Xi)$\textup{)}.
\end{theorem}

\subsection{Reduced Gr\"{o}bner Basis and Timings}


Now, it's time for our improvement for Farr-Gao algorithm.


\begin{alg}\label{XTBM}

\textbf{Input}: A finite set $\Xi\subset \mathbb{F}^2$ and a fixed monomial order $\prec$ that is either $\prec_{\mathrm{grlex}}$ or $\prec_{\mathrm{lex}}$ \textup{(}resp. either $\prec_{\mathrm{grevlex}}$ or $\prec_{\mathrm{inlex}}$\textup{)}.


\textbf{Output}: The reduced Gr\"{o}bner basis for $\mathcal{I}(\Xi)$ w.r.t. $\prec$.


\indent \textbf{Step1.} Decompose $\Xi$ following \eqref{e:decom} and find an $x$-tower \textup{(}resp. $y$-tower\textup{)} subset $\Xi_\mathrm{T}$ of $\Xi$ as large as possible.
\\
\indent \textbf{Step2.} Obtain $S_x(\Xi_\mathrm{T})$ \textup{(}resp. $S_y(\Xi_\mathrm{T})$\textup{)} following \eqref{e:sxsy}, and finally express $\Xi_\mathrm{T}$ in form \eqref{e:towerx} \textup{(}resp. \eqref{e:towery}\textup{)}.
\\
\indent \textbf{Step3.} Construct list $P$ whose $h$-th entry $p_h$, $1\leq h\leq \#\Xi_\mathrm{T}$, is the $h$-th smallest element of $\Xi_\mathrm{T}$ in form \eqref{e:towerx} \textup{(}resp. \eqref{e:towery}\textup{)} w.r.t. increasing $\prec_{\mathrm{inlex}}$\textup{(}resp. $\prec_{\mathrm{lex}}$\textup{)} on $(i, j)$.
\\
\indent \textbf{Step4.} Compute set $M:=\mathfrak{M}_{S_x(\Xi_\mathrm{T})}$ \textup{(}resp. $\mathfrak{M}_{S_y(\Xi_\mathrm{T})}$\textup{)} and then set $C:= \mathcal{C}[M]$ by applying \eqref{e:lowerbas} and \eqref{e:corner} respectively.
\\
\indent \textbf{Step5.} Construct list $N$ whose $k$-th entry $q_k$, $1\leq k\leq \#\Xi_\mathrm{T}$, is the $k$-th smallest element of $\mathfrak{N}_{S_x(\Xi_\mathrm{T})}$ \textup{(}resp. $\mathfrak{N}_{S_y(\Xi_\mathrm{T})}$\textup{)} w.r.t. increasing $\prec_{\mathrm{inlex}}$\textup{(}resp. $\prec_{\mathrm{lex}}$\textup{)} on $(i, j)$.
\\
\indent \textbf{Step6.} Use $C, N$ to obtain the reduced Gr\"{o}bner basis $G$ for $\mathcal{I}(\Xi_\mathrm{T})$.
\\
\indent \textbf{Step7.} Send $G$ to Farr-Gao process to finish the computation.
\end{alg}

Algorithm \ref{XTBM} has been implemented on \verb"Maple"$\circledR$ 16 that is installed on a laptop with 8 Gb RAM and 2.3 GHz CPU. For $\prec_{\mathrm{grevlex}}$, its running times on 250, 500, and 1000 points in $\mathbb{F}_q^2$ are compared with the build-in command \verb"VanishingIdeal" of \verb"Maple"$\circledR$.

\vskip 0.3cm

When $q=41$, we have

\begin{center}
\begin{tabular}{l|lllll}\hline
  \backslashbox{Algorithms}{$\#\Xi$} & 250     & 500   & 1000  \\ \hline\hline
  Algorithm \ref{XTBM}      & 4.264 s & 29.002 s & 142.377 s \\
  \verb"VanishingIdeal"     & 4.883 s & 31.675 s  & 147.515 s\\
  \hline
\end{tabular}
\end{center}

\vskip 0.3cm

When $q=101$, we have

\vskip 0.3cm

\begin{center}
\begin{tabular}{l|lllll}\hline
  \backslashbox{Algorithms}{$\#\Xi$} & 250     & 500   & 1000  \\ \hline\hline
  Algorithm \ref{XTBM}      & 4.384 s & 29.998 s & 146.227 s \\
  \verb"VanishingIdeal"     & 4.961 s & 31.684 s  & 151.758 s\\
  \hline
\end{tabular}
\end{center}

\bibliographystyle{plain}
\bibliography{refDCZL2010}

\begin{thebibliography}{10}

\bibitem{CLO2007}
David Cox, John Little, and Donal O'Shea.
\newblock {\em Ideal, Varieties, and Algorithms}.
\newblock Undergrad. Texts Math. Springer, New York, 3 edition, 2007.

\bibitem{Cra2004}
N.~Crainic.
\newblock Multivariate {Birkhoff-Lagrange} interpolation schemes and cartesian
  sets of nodes.
\newblock {\em Acta Math. Univ. Comenian.(N.S.)}, LXXIII(2):217--221, 2004.

\bibitem{dBo2007}
Carl de~Boor.
\newblock Interpolation from spaces spanned by monomials.
\newblock {\em Adv. Comput. Math.}, 26(1):63--70, 2007.

\bibitem{DZL2005}
Tian Dong, Shugong Zhang, and Na~Lei.
\newblock Interpolation basis for nonuniform rectangular grid.
\newblock {\em J. Inf. Comput. Sci.}, 2(4):671--680, 2005.

\bibitem{FG2006}
Jeffrey Farr and Shuhong Gao.
\newblock Computing {Gr\"{o}bner} bases for vanishing ideals of finite sets of
  points.
\newblock In M.~Fossorier, H.~Imai, S.~Lin, and A.~Poli, editors, {\em Applied
  Algebra, Algebraic Algorithms and Error-Correcting Codes}, volume 3857 of
  {\em Lecture Notes in Computer Science}, pages 118--127. Springer, Berlin,
  2006.

\bibitem{GS1999}
V.~Guruswami and M.~Sudan.
\newblock Improved decoding of reed-solomon and algebraicgeometric codes.
\newblock {\em IEEE Trans. Inform. Theory}, 46(6):1757¨C--1767, 1999.

\bibitem{ABKR2000}
J.Abbott, A.Bigatti, M.Kreuzer, and L.Robbiano.
\newblock Computing ideals of points.
\newblock {\em J. Symbolic Comput.}, 30:341--356, 2000.

\bibitem{LS2004}
Reinhard Laubenbacher and Brandilyn Stigler.
\newblock A computational algebra approach to the reverse engineering of gene
  regulatory networks.
\newblock {\em J. Theoret. Biol.}, 229(4):523--537, 2004.

\bibitem{LZD2012}
Zhe Li, Shugong Zhang, and Tian Dong.
\newblock Finite sets of affine points with unique associated monomial order
  quotient bases.
\newblock {\em J. Algebra Appl.}, 11(2):1250025, 2012.

\bibitem{MMM1993}
M.~G. Marinari, H.~Michael M\"{o}ller, and T.~Mora.
\newblock Gr\"{o}bner bases of ideals defined by functionals with an
  application to ideals of projective points.
\newblock {\em Appl. Algebra Engrg. Comm. Comput.}, 4(2):103--145, 1993.

\bibitem{MB1982}
H.~M\"{o}ller and B.~Buchberger.
\newblock The construction of multivariate polynomials with preassigned zeros.
\newblock In J.~Calmet, editor, {\em Computer Algebra: EUROCAM '82}, volume 144
  of {\em Lecture Notes in Computer Science}, pages 24--31. Springer, Berlin,
  1982.

\bibitem{Mor2009}
Teo Mora.
\newblock Gr\"{o}bner technology.
\newblock In Massimiliano Sala, Teo Mora, Ludovic Perret, Shojiro Sakata, and
  Carlo Traverso, editors, {\em Gr\"{o}bner Bases, Coding, and Cryptography},
  pages 11--25. Springer, Berlin, 2009.

\bibitem{Sau2004}
Thomas Sauer.
\newblock Lagrange interpolation on subgrids of tensor product grids.
\newblock {\em Math. Comp.}, 73(245):181--190, 2004.

\bibitem{Sau2006}
Thomas Sauer.
\newblock Polynomial interpolation in several variables: Lattices, differences,
  and ideals.
\newblock In K.~Jetter, M.~Buhmann, W.~Haussmann, R.~Schaback, and
  J.~St\"{o}ckler, editors, {\em Topics in Multivariate Approximation and
  Interpolation}, volume~12 of {\em Studies in Computational Mathematics},
  pages 191--230. Elsevier, Amsterdam, 2006.

\bibitem{WZD2010}
Xiaoying Wang, Shugong Zhang, and Tian Dong.
\newblock A bivariate preprocessing paradigm for the {B}uchberger-{M}\"{o}ller
  algorithm.
\newblock {\em J. Comput. Appl. Math.}, 234(12):3344--3355, 2010.

\end{thebibliography}

\end{document}